\documentclass[12pt]{amsart}
\usepackage{amsthm}
\usepackage{amssymb}
\usepackage{cite}
\usepackage{cases}
\usepackage{longtable}
\usepackage{geometry}
\usepackage{enumerate}
\usepackage{setspace}
\usepackage{tikz-cd}
\usetikzlibrary{matrix, arrows}
\usepackage[unicode=true]
 {hyperref}
\hypersetup{
colorlinks=true,
urlcolor=black,
citecolor=blue,
linkcolor=blue,
}

\allowdisplaybreaks
\newtheorem{thm}{Theorem}[section]
\newtheorem{prop}[thm]{Proposition}
\newtheorem{lem}[thm]{Lemma}
\newtheorem{cor}[thm]{Corollary}

\theoremstyle{definition}
\newtheorem{definition}[thm]{Definition}

\newtheorem{example}[thm]{Example}

\theoremstyle{remark}
\newtheorem{remark}[thm]{Remark}

\numberwithin{equation}{section}

\newcommand{\Ann}{\mathrm{Ann}}
\newcommand{\Aut}{\mathrm{Aut}}

\newcommand{\op}{\mbox{\tiny op}}

\begin{document}

\large 

\title{On Gr\"{u}n's lemma for perfect skew braces}
\author{Cindy (Sin Yi) Tsang}
\address{Department of Mathematics\\
Ochanomizu University\\
2-1-1 Otsuka, Bunkyo-ku\\
Tokyo\\
Japan}
\email{tsang.sin.yi@ocha.ac.jp}\urladdr{http://sites.google.com/site/cindysinyitsang/} 
\date{\today}

\maketitle

\begin{abstract}By previous work of Ced\'{o}, Smoktunowicz, and Vendramin, one already knows that the analog of Gr\"{u}n's lemma fails to hold for perfect skew left braces when the socle is used as an analog of the center of a group. In this paper, we use the annihilator instead of the socle. We shall show that the analog of Gr\"{u}n's lemma holds for perfect two-sided skew braces but not in general.
\end{abstract}

\tableofcontents
 
\section{Introduction}

A \emph{skew brace} is any set $A=(A,\cdot,\circ)$ equipped with two group operations $\cdot$ and $\circ$ such that left brace relation
\[ a\circ (b\cdot c) = (a\circ b)\cdot a^{-1}\cdot (a\circ c)\]
holds for all $a,b,c\in A$, where $a^{-1}$ denotes the inverse of $a$ with respect to $\cdot$. It is easy to see that $(A,\cdot)$ and $(A,\circ)$ must share the same identity element, which we denote by $1$. Due to their relations with the set-theoretic solutions to the Yang--Baxter equation, understanding the structure of skew braces is a problem of interest; we refer the reader to \cite{Rump,GV,SB,ann} for more details.

Given a group $(A,\cdot)$, we can construct a skew brace $(A,\cdot,\circ)$ by defining $\circ$ to be the same operation $\cdot$, or its opposite operation $\cdot^{\op}$, that is $a\cdot^{\op}b = b\cdot a$. The skew braces of the forms $(A,\cdot,\cdot)$ and $(A,\cdot,\cdot^{\op})$ are said to be \emph{trivial} and \emph{almost trivial}, respectively, because they are essentially just groups. 

Skew braces may therefore be regarded as an extension of groups. Indeed, there are many similarities between skew braces and groups (see \cite{Tsa,isoclinism,Schur,Lazard} for some examples). The purpose of this paper is to continue research in this direction and explore analogs of Gr\"{u}n's lemma  \cite[Satz 4]{Grun} for skew braces.

A group $G$ is said to be \textit{perfect} if it equals its derived subgroup $[G,G]$. For any group $G$, let us denote its center by $Z(G)$.

\begin{thm}[Gr\"{u}n's lemma]\label{thm:Grun}For any perfect group $G$, we have
\[Z(G/Z(G))=1.\]
\end{thm} 

To consider the analog of Gr\"{u}n's lemma in the context of skew braces, we need to first define ``quotient", ``perfect", and ``center" for skew braces.

In what follows, let $A = (A,\cdot,\circ)$ be a skew brace. 

\begin{definition}\label{def:quotient}A subset $I$ of $A$ is said to be an \textit{ideal} if 
\begin{enumerate}[(1)]
\item $I$ is a normal subgroup of $(A,\cdot)$;
\item $I$ is a normal subgroup of $(A,\circ)$;
\item $a\cdot I=a\circ I$ for all $a\in A$.
\end{enumerate}
In this case, we can naturally endow the coset space
\[ A/I = \{a\cdot I : a\in A\} = \{a\circ I: a\in A\}\]
with a quotient skew brace structure from that of $A$.  
\end{definition}

To measure the difference between the group operations $\cdot$ and $\circ$, define  
\[ a * b = a^{-1} \cdot (a\circ b) \cdot b^{-1}.\]
For example, we have $a*b=1$ when $A$ is trivial, and $a*b = a^{-1}\cdot b \cdot a\cdot b^{-1}$ when $A$ is almost trivial. We then see that $*$ may be viewed as an analog of the commutator $[\, ,\, ]$, and the notion of ``trivial" for skew braces is a natural analog of ``abelian" for groups. For any subsets $X,Y$ of $A$, we define $X*Y$ to be the subgroup of $(A,\cdot)$ generated by the elements $x*y$ for $x\in X,\, y\in Y$. 

\begin{definition}\label{def:perfect} The \textit{derived ideal} of $A$ is defined to be the subset $A*A$. It is known (see \cite[Proposition 2.1]{perfect}) that $A*A$ is indeed an ideal of $A$, in fact the smallest ideal of $A$ for which the quotient skew brace is trivial. We shall say that $A$ is \textit{perfect} if it equals its derived ideal $A*A$.
\end{definition}

As for the analog of ``center", there are two natural candidates.

\begin{definition}\label{def:center} The \textit{socle} of $A$ is defined as
\[ \mathrm{Soc}(A) = \{a\in A \mid \forall x\in A: a*x=1\}\cap Z(A,\cdot),\]
which is an ideal of $A$ by \cite[Lemma 2.5]{GV}. The \textit{annihilator} of $A$ is defined as
\begin{align*}
\Ann(A) &= \{a\in A \mid \forall x\in A: a*x=1\}\cap Z(A,\cdot)\cap Z(A,\circ)\\
&= \{a\in A\mid \forall x\in A: a*x=1=x*a\}\cap Z(A,\cdot), \end{align*}
which is easily checked to be an ideal of $A$.
\end{definition}

As analogs of Gr\"{u}n's lemma, it is natural to ask whether
\[\mathrm{Soc}(A/\mathrm{Soc}(A))=1\quad\mbox{and}\quad \Ann(A/\Ann (A))=1\]
hold for all perfect skew braces $A$. It is already known by \cite[Section 3]{perfect} that there exist perfect skew braces $A$ for which the former equality fails. In this paper, we wish to consider the latter equality instead. Note that
\[ \Ann(A/\Ann(A)) = \Ann_2(A)/\Ann(A)\]
for a unique ideal $\Ann_2(A)$ of $A$ by the isomorphism theorems in groups. We shall refer to $\Ann_2(A)$ as the \textit{second annihilator} of $A$. By Definition \ref{def:center}, we are then reduced to investigating whether
\begin{align}\label{good1}
\Ann_2(A)*(A*A)&=1,\\\label{good2}
[\Ann_2(A),A*A]&=1,\\\label{bad}
(A*A)*\Ann_2(A)&=1,
\end{align}
where $[\,,\, ]$ denotes the commutator in the group $(A,\cdot)$. In the case that $A$ is perfect, we have $A=A*A$, so then (\ref{good1}), (\ref{good2}), and (\ref{bad}) would imply
\[\Ann_2(A)=\Ann(A)\quad \mbox{that is} \quad \Ann(A/\Ann(A))=1.\]
However, as we shall show, while the equalities (\ref{good1}) and (\ref{good2}) always hold, the equality (\ref{bad}) fails in some cases.

\begin{definition} We shall say that $A$ is \textit{two-sided} if the right brace relation
\[ (b\cdot c)\circ a = (b\circ a)\cdot a^{-1}\cdot (c\circ a)\]
also holds for all $a,b,c\in A$.
\end{definition}

Our main results are as follows:

\begin{thm}\label{thm1}For any skew brace $A$, we have
\[ \Ann_2(A)*(A*A)=1\quad\mbox{and}\quad[\Ann_2(A),A*A]=1.\]
\end{thm}
\begin{proof}See Propositions \ref{Prop1} and \ref{Prop2}.
\end{proof}

\begin{thm}\label{thm2}For any two-sided skew brace $A$, we have
\[ (A*A)*\Ann_2(A)=1.\]
\end{thm}
\begin{proof} See Corollary \ref{cor}.
\end{proof}

\begin{cor}\label{cor'}For any two-sided perfect skew brace $A$, we have
\[ \Ann(A/\Ann(A))=1.\]
\end{cor}
\begin{proof} This follows immediately from Theorems \ref{thm1} and \ref{thm2}.\end{proof}

Note that Theorem \ref{thm:Grun} may be recovered from Corollary \ref{cor'} by taking $A$ to be an almost trivial skew brace. Thus, Corollary \ref{cor'} is a genuine generalization of Gr\"{u}n's lemma for skew braces.

In Section \ref{sec:example}, we shall describe ways to construct skew braces $A$ for which 
\[(A*A)*\Ann_2(A)\neq 1.\]
 We shall also see that $A$ may be chosen to be perfect, so as a consequence, we obtain that:

\begin{cor}There exist perfect skew braces $A$ for which
\[ \Ann(A/\Ann(A))\neq 1.\]
\end{cor}
\begin{proof}This follows from Proposition \ref{prop3} and Examples \ref{example1}  and \ref{example2}.
\end{proof}
 
\section{Derived ideal and second annihilator}

In this section, let $A=(A,\cdot,\circ)$ be a skew brace. For each $a\in A$, define
\[ \lambda_a : A \longrightarrow A;\,\ \lambda_a(x)=a^{-1}\cdot (a\circ x),\]
which is easily checked to be an element of $\Aut(A,\cdot)$. The map
\[ \lambda : (A,\circ) \longrightarrow \Aut(A,\cdot);\,\ a\mapsto \lambda_a\]
is a morphism of groups (see \cite[Proposition 1.9]{GV}). For any $a,b\in A$, in terms of this map $\lambda$, we have the well-known identities
\[ a \circ b = a\cdot \lambda_a(b),\quad a\cdot b = a \circ \lambda_{\overline{a}}(b),\]
\[\overline{a} = \lambda_{\overline{a}}(a^{-1}),\quad a*b = \lambda_a(b)\cdot b^{-1},\]
where $\overline{a}$ denotes the inverse of $a$ with respect to $\circ$. Let us also write
\[ [a,b]  = a\cdot b\cdot a^{-1}\cdot b^{-1},\quad
[a,b]_\circ = a\circ b\circ \overline{a}\circ \overline{b}\]
for the commutators in the groups $(A,\cdot)$ and $(A,\circ)$, respectively. 
The following identities shall also be useful.

\begin{lem} For any $a,x,y\in A$, we have
\begin{align}
\label{iden1}
a* (x\cdot y) &= (a*x) \cdot x \cdot (a*y)\cdot x^{-1},\\
\label{iden2} 
(x\circ y) * a &= (x*(y*a))\cdot (y*a) \cdot (x*a),\\
\label{iden3}
\lambda_a(x*y) & = (a\circ x \circ \overline{a}) * \lambda_a(y).
\end{align}
\end{lem}
\begin{proof}
Straightforward or see \cite[Lemmas 2.1 and 2.3]{Tsa}.
\end{proof}

In the following, we shall investigate the relationship between the derived ideal $A*A$ and the second annihilator $\Ann_2(A)$. In particular, we shall show that (\ref{good1}) and (\ref{good2}) always hold, and then give a characterization of (\ref{bad}).

\subsection{$\Ann_2(A)*(A*A)=1$} 
For each $a\in A$, consider the map
\[ \varphi_a: (A,\cdot) \longrightarrow (A*A,\cdot);\,\ \varphi_a(x) = a*x.\]
In the case that $a\in \Ann_2(A)$, we have $\mathrm{Im}(\varphi_a) \subseteq \Ann(A) \subseteq Z(A,\cdot)$, and we deduce from (\ref{iden1}) that $\varphi_a$ is a group homomorphism.

\begin{prop}\label{Prop1}We have $\Ann_2(A)*(A*A)=1$.
\end{prop}
\begin{proof} It suffices to check that $a*z=1$ for all $a\in \Ann_2(A)$ and $z\in A*A$. In other words, we want to show that $A*A\subseteq\ker(\varphi_a)$ for each $a\in \Ann_2(A)$. Since $\varphi_a$ is a homomorphism on $(A,\cdot)$, we only need to show that
\[ x*y\in\ker(\varphi_a)\mbox{ for all } x,y\in A\]
because these elements $x*y$ generate $A*A$ in $(A,\cdot)$. 

Let $a\in \Ann_2(A)$ and $x,y\in A$. Note that
\[ \varphi_a(x*y) = \varphi_a(\lambda_x(y))\cdot \varphi_a( y)^{-1} =( a*\lambda_x(y)) \cdot (a*y)^{-1}.\]
Since $[\overline{a},x]_\circ \in \Ann(A)$, we have $[\overline{a},x]_\circ * \lambda_x(y)=1$ and (\ref{iden2}) yields that
\[ a * \lambda_x(y)= (a \circ [\overline{a},x]_\circ ) * \lambda_x(y)
 = (x \circ a \circ \overline{x})*\lambda_x(y),\]
which in turn is equal to $\lambda_x(a*y)$ by (\ref{iden3}). Hence, we obtain
\[ \varphi_a(x*y) =  \lambda_x(a*y)\cdot (a*y)^{-1} = x * (a*y),\]
which is equal to $1$ because $a*y\in \Ann(A)$. This completes the proof.
\end{proof}

\subsection{$[\Ann_2(A),A*A]=1$} For each $a\in A$, consider the map
\[ \pi_a: (A,\cdot)\longrightarrow ([A,A],\cdot);\,\ \pi_a(x) = [a,x].\]
In the case that $a\in \Ann_2(A)$, we have $\mathrm{Im}(\pi_a) \subseteq \Ann(A) \subseteq Z(A,\cdot)$, and we deduce from the standard identity 
\[ [a,x\cdot y]=[a,x]\cdot x\cdot [a,y]\cdot x^{-1}\]
(note that (\ref{iden1}) is an analog of this) that $\pi_a$ is a group homomorphism.

\begin{prop}\label{Prop2} We have $[\Ann_2(A),A*A]=1$.
\end{prop}
\begin{proof}
It is enough to show that $[a,z]=1$ for all $a\in \Ann_2(A)$ and $z\in A*A$. In other words, we want to show that $A*A\subseteq\ker(\pi_a)$ for each $a\in \Ann_2(A)$. Since $\pi_a$ is a homomorphism on $(A,\cdot)$, we only need to check that
\[ x*y\in\ker(\pi_a)\mbox{ for all } x,y\in A\]
because these elements $x*y$ generate $A*A$ in $(A,\cdot)$. 

Let $a\in \Ann_2(A)$ and $x,y\in A$. Note that
\[ \pi_a(x*y) = \pi_a(\lambda_x(y))\cdot \pi_a(y)^{-1} = [a,\lambda_x(y)] \cdot [a,y]^{-1}.\]
Since $x*a\in \Ann(A)\subseteq Z(A,\cdot)$, we have
\[ [a,\lambda_x(y)] = [(x*a)a,\lambda_x(y)] = [\lambda_x(a),\lambda_x(y)],\]
which in turn is equal to $\lambda_x([a,y])$ because $\lambda_x\in\Aut(A,\cdot)$. It follows that
\[ \pi_a(x*y) = \lambda_x([a,y])\cdot [a,y]^{-1} = x * [a,y],\]
which is equal to $1$ because $[a,y]\in \Ann(A)$. This completes the proof.
\end{proof}

\subsection{$(A*A)*\Ann_2(A)=1$} For each $a\in A$, consider the map
\[ \psi_a : (A,\circ) \longrightarrow (A*A,\cdot);\,\ \psi_a(x) = x*a.\]
Note that unlike the $\phi_a$ and $\pi_a$ considered in the previous subsections, here we use the group operation $\circ$ in domain. In the case that $a\in \Ann_2(A)$, we have $\mathrm{Im}(\psi_a)\subseteq \Ann(A)\subseteq Z(A,\cdot)$, and we see from (\ref{iden2}) that $\psi_a$ is a group homomorphism since $A*\Ann(A)=1$. 
 
Now, it is obvious that
\begin{align}\label{char}
(A*A)*\Ann_2(A) = 1& \iff \forall a\in \Ann_2(A),\, z\in A*A: z*a = 1\\\notag
& \iff \forall a\in \Ann_2(A): A*A \subseteq \ker(\psi_a).
\end{align}
Let us consider when this last inclusion is satisfied.

\begin{prop}\label{char1}For each $a\in \Ann_2(A)$, we have
\[ \psi_a\mbox{ is a homomorphism on $(A,\cdot)$}\,\ \iff \,\ A*A \subseteq \ker(\psi_a).\]
 \end{prop}
\begin{proof}First, suppose that $\psi_a$ is a homomorphism on $(A,\cdot)$, Then it suffices to check that $x*y\in \ker(\psi_a)$ for all $x,y\in A$ since these elements $x*y$ generate $A*A$ in $(A,\cdot)$. But $\psi_a$ is also a homomorphism on $(A,\circ)$, so clearly
\begin{align*}
\psi_a(x*y) & = \psi_a(x^{-1}\cdot (x\circ y)\cdot y^{-1}) \\
& = \psi_a(x)^{-1}\cdot\psi(x)\cdot \psi_a(y)\cdot\psi_a(y)^{-1}\\
& =1.
\end{align*}
Conversely, suppose that $A*A\subseteq \ker(\psi_a)$. For any $x,y\in A$, first we write
\[\psi_a(x\cdot y) = \psi_a(x\circ \lambda_{\overline{x}}(y)) = \psi_a(x) \cdot \psi_a(\lambda_{\overline{x}}(y)).\]
For the second term, we compute that
\begin{align*}
\psi_a(\lambda_{\overline{x}}(y)) & = \psi_a(y) \cdot \psi_a(y)^{-1}\cdot \psi_a(\lambda_{\overline{x}}(y)) \\
&= \psi_a(y) \cdot \psi_a(\overline{y}\circ \lambda_{\overline{x}}(y))\\
& = \psi_a(y)\cdot \psi_a(\overline{y}\cdot\lambda_{\overline{y}\circ \overline{x} \circ y}(\overline{y}^{-1}))\\
& = \psi_a(y)\cdot \psi_a( \overline{y}\cdot ((\overline{y}\circ \overline{x}\circ y) * \overline{y}^{-1})\cdot \overline{y}^{-1})\\
& =\psi_a(y),
\end{align*}
where last equality holds because $A*A$ is normal in $(A,\cdot)$ and so
\[\overline{y}\cdot ((\overline{y}\circ \overline{x}\circ y) * \overline{y}^{-1})\cdot \overline{y}^{-1} \in A*A \subseteq \ker(\psi_a)\]
(see \cite[Corollary 2.2]{Tsa} for example). Thus, we have $\psi_a(x\cdot y) = \psi_a(x)\cdot \psi_a(y)$, whence $\psi_a$ is a homomorphism on $(A,\cdot)$.
\end{proof}

For each $a\in A$, consider the inner automorphism
\[ \iota_a : (A,\circ) \longrightarrow (A,\circ) ;\,\ \iota_a(x) = a\circ x\circ\overline{a}.\]
It is basically known in the literature (see \cite[Lemma 4.1]{TN} or \cite[Proposition 2.3]{two sided} for example) that for any $x,y\in A$, we have 
\[ \iota_a(x\cdot y)=\iota_a(x)\cdot \iota_a(y) \,\ \iff \,\  (x\cdot y)\circ \overline{a} = (x\circ \overline{a})\cdot \overline{a}^{-1}\cdot(y\circ \overline{a}).\]
In other words, we have $\iota_a\in \Aut(A,\cdot)$ if and only if the right brace relation holds when the element on the right of $\circ$ is fixed to be $\overline{a}$. Let us now show that the two equivalent conditions in Proposition \ref{char1} may be characterized in terms of these inner automorphisms.

\begin{prop}\label{char2} For each $a\in \Ann_2(A)$, we have the relation
\[ \iota_a(x) = a\lambda_a(x)a^{-1}\cdot\psi_{\overline{a}}(x)\]
for all $x\in A$. Moreover, we have the equivalence
\[ \psi_{\overline{a}}\mbox{ is a homomorphism on $(A,\cdot)$}\,\ \iff \,\ \iota_a\in \Aut(A,\cdot).\]
\end{prop}
\begin{proof} For the equality, we first compute that
\begin{align*}
\iota_a(x) & = a\circ x\circ \overline{a}\\
&= a\cdot \lambda_a(x\cdot (x*\overline{a})\cdot \overline{a})\\
& = a\cdot \lambda_a(x)\cdot  \lambda_a(x*\overline{a})\cdot \lambda_a(\overline{a})\\
& = a\cdot \lambda_a(x)\cdot (a*(x*\overline{a}))\cdot (x*\overline{a})\cdot a^{-1}.
\end{align*}
Since $x*\overline{a}\in \Ann(A)\subseteq Z(A,\cdot)$, we can move it to the right of $a^{-1}$, and also the third term vanishes because $A*\Ann(A)=1$. We thus obtain
\[ \iota_a(x) = a\lambda_a(x)a^{-1}\cdot \psi_{\overline{a}}(x),\]
as claimed. For any $x,y\in A$, we then get that
\begin{align*}
\iota_a(x\cdot y) & =a\lambda_a(x\cdot y)a^{-1}\cdot \psi_{\overline{a}}(x\cdot y)\\
&=a\lambda_a(x)\lambda_a(y)a^{-1}\cdot \psi_{\overline{a}}(x\cdot y),\\
\iota_a(x)\cdot \iota_a(y)& =a\lambda_a(x)a^{-1}\cdot \psi_{\overline{a}}(x) \cdot a\lambda_a(y)a^{-1}\cdot \psi_{\overline{a}}(y)\\
&=a\lambda_a(x)\lambda_a(y)a^{-1}\cdot \psi_{\overline{a}}(x)\cdot\psi_{\overline{a}}(y),
\end{align*}
where we have again used the fact that $\psi_{\overline{a}}(x) = x*\overline{a}\in  Z(A,\cdot)$. Thus
\[ \iota_a(x\cdot y) = \iota_a(x)\cdot \iota_a(y) \,\ \iff \,\ \psi_{\overline{a}}(x\cdot y) = \psi_{\overline{a}}(x)\cdot \psi_{\overline{a}}(y),\]
and this proves the equivalence.
\end{proof}

\begin{cor}\label{cor} For any skew brace $A$, we have
\[ (A*A)*\Ann_2(A) = 1 \,\ \iff \,\ \forall a\in \Ann_2(A) : \iota_a\in \Aut(A,\cdot).\]
In particular, we have $(A*A)*\Ann_2(A)=1$ whenever $A$ is two-sided.
\end{cor}
\begin{proof} This follows from  (\ref{char}) and Propositions \ref{char1} and \ref{char2}.
\end{proof}

\section{Constructing counterexamples}\label{sec:example}

In this section, we use semidirect products to construct examples of skew braces $A$ for which $(A*A)*\Ann_2(A)\neq 1$. Such skew braces are necessarily non-two-sided by Corollary \ref{cor}. We shall also show that $A$ may be chosen to be perfect, thus yielding counterexamples to the analog of Gr\"{u}n's lemma.

\begin{prop}\label{semi}Let $B = (B,\cdot,\circ)$ and $C=(C,\cdot,\circ)$ be skew braces. Let
\[ \phi : (C,\circ) \longrightarrow \Aut(B,\cdot)\cap \Aut(B,\circ);\,\ c\mapsto \phi_c\]
be any group homomorphism and define
\begin{align*}
(b_1,c_1)\cdot (b_2,c_2) & = (b_1\cdot b_2, c_1\cdot c_2) \\
(b_1,c_1)\circ (b_2,c_2)&= (b_1\circ \phi_{c_1}(b_2),c_1\circ c_2)
\end{align*}
on the set $B\times C$. Then $(B\times C,\cdot,\circ)$ is a skew brace. Moreover, we have
\begin{equation}\label{* iden} (b_1,c_1) * (b_2,c_2) = (b_1^{-1}\cdot (b_1\circ \phi_{c_1}(b_2))\cdot b_2^{-1},c_1 *c_2)\end{equation}
for any $b_1,b_2\in B$ and $c_1,c_2\in C$.
\end{prop}

The skew brace $(B\times C,\cdot,\circ)$ constructed above is denoted by $B\rtimes_\phi C$ and is called the \textit{semidirect product} of $B$ and $C$ with respect to $\phi$. 

\begin{proof} This construction is well-known (mentioned in \cite[Corollary 2.37]{SB} for example), and the proof is straightforward. As for (\ref{* iden}), we compute that
\begin{align*}
(b_1,c_1) * (b_2,c_2)& = (b_1,c_1)^{-1}\cdot ((b_1,c_1)\circ (b_2,c_2)) \cdot (b_2,c_2)^{-1}\\
&= (b_1^{-1},c_1^{-1})\cdot (b_1\circ \phi_{c_1}(b_2),c_1\circ c_2)\cdot (b_2^{-1},c_2^{-1})\\
&= (b_1^{-1}\cdot (b_1\circ \phi_{c_1}(b_2))\cdot b_2^{-1},c_1 *c_2),
\end{align*}
which is as claimed.
 \end{proof}

We now specialize to the case when $B = (B,+,+)$ is a trivial \textit{brace}, which means that the operation $+$ is taken to be commutative. Here we denote the operation of  $B$ by $+$ because we shall take $B$ to be a vector space to produce explicit examples. Note that in this case, the identity (\ref{* iden}) becomes
\begin{equation}\label{*B} (b_1,c_1) * (b_2,c_2) = ((\phi_{c_1}-\mathrm{id}_B)(b_2),c_1 *c_2)\end{equation}
for any $b_1,b_2\in B$ and $c_1,c_2\in C$. We shall also take $C$ to be perfect. 

\begin{prop}\label{sd prop}Let $B = (B,+,+)$ be a trivial brace and let $C = (C,\cdot,\circ)$ be a perfect skew brace. Consider the semidirect product $A = B\rtimes_\phi C$, where
\[ \phi : (C,\circ)\longrightarrow \Aut(B,+);\,\ c\mapsto \phi_c\]
is any group homomorphism. Then we have the equality 
\begin{equation}\label{A*A}
 A*A = \Bigg\langle \bigcup_{\xi\in\mathrm{Im}(\phi)}\mathrm{Im}(\xi - \mathrm{id}_B)\Bigg\rangle\times C  ,
 \end{equation}
 and also the inclusions
\[\mathrm{Fix}_{\phi}(B) \times \{1\}\subseteq \Ann(A)\subseteq \mathrm{Fix}_{\phi}(B) \times (\ker(\phi)\cap \Ann(C)),\]
where $\mathrm{Fix}_{\phi}(B)$ is the fixed-point subgroup of $\mathrm{Im}(\phi)$, namely
\[ \mathrm{Fix}_{\phi}(B) = \bigcap_{\xi\in \mathrm{Im}(\phi)} \ker(\xi - \mathrm{id}_B).\]
In particular, we have
\begin{equation}\label{Ann(A)} \Ann(A) = \mathrm{Fix}_{\phi}(B) \times \{1\}\mbox{ when }\Ann(C)=1.\end{equation}
\end{prop}
\begin{proof}From (\ref{*B}), it is clear that the left-to-right inclusion of (\ref{A*A}) holds.  By taking $c_2=1$ and $b_2=0$, respectively, we also see that 
\[  ((\phi_{c_1}-\mathrm{id}_B)(b_2),1),(0,c_1*c_2)\in A*A.\]
Since $b_2\in B,\, c_1,c_2\in C$ are arbitrary and $C = C*C$, it follows that
\[ \Bigg\langle\bigcup_{\xi\in\mathrm{Im}(\phi)}\mathrm{Im}(\xi-\mathrm{id}_B)\Bigg\rangle\times \{1\}, \, \{0\} \times C \subseteq A*A.\]
This shows that the right-to-left inclusion of (\ref{A*A}) also holds.

Next, consider an element $(b,c)$ from the center $Z(A,\cdot) = (B,+)\times Z(C,\cdot)$ of $(A,\cdot)$. By definition, we have $(b,c)\in \Ann(A)$ if and only if
\begin{align*}
(b,c)* (x,y) & = ((\phi_c-\mathrm{id}_B)(x),c*y) = (0,1)\\
(x,y)* (b,c) & = ((\phi_y-\mathrm{id}_B)(b),y*c) = (0,1)
\end{align*}
are satisfied for all $x\in B, y\in C$. They obviously hold when $b\in \mathrm{Fix}_{\phi}(B)$ and $c =1$. Conversely, they imply the equalities
\begin{align*}
 \forall y\in Y:&\,  c*y=y*c=1,\hspace{-1cm}&&\mbox{which yields }c\in \Ann(C),\\
 \forall x\in X:&\, (\phi_c-\mathrm{id}_B)(x) = 0,\hspace{-1cm}&&\mbox{which yields }c\in \ker(\phi),\\
 \forall y\in Y:&\, (\phi_y-\mathrm{id}_B)(b) =0,\hspace{-1cm}&&\mbox{which yields }b\in \mathrm{Fix}_{\phi}(B).
 \end{align*}
 The inclusions regarding $\Ann(A)$ then follow as well. 
 \end{proof}
  
To simplify things further, we shall specialize to the case when
\[ B = (\mathbb{F}_p^n,+,+)\quad \mbox{and}\quad \Ann(C)=1,\]
where $p$ is a prime and $n\geq 2$ is a positive integer. Below, we shall give three ways to construct the homomorphism $\phi$ so that 
\begin{equation}\label{neq}
(A*A)*\Ann_2(A)\neq 1\mbox{ for  }A = B\rtimes_\phi C.
\end{equation}
In view of (\ref{A*A}) and (\ref{Ann(A)}), the only thing that matters is $\mathrm{Im}(\phi)$,  and for each $n=2,3,4$, we shall give an example of $\mathrm{Im}(\phi)$ that would realize (\ref{neq}).



\begin{prop}\label{prop1} 
Let $C = (C,\cdot,\circ)$ be a perfect skew brace with $\Ann(C) = 1$ and let $A = \mathbb{F}_p^2 \rtimes_\phi C$, where $p$ is any prime and $\phi : (C,\circ) \longrightarrow \mathrm{GL}_2(\mathbb{F}_p)$ is a homomorphism with
\[ \mathrm{Im}(\phi)=\langle  \left[\begin{smallmatrix}
1 & 1 \\
0 & 1
\end{smallmatrix}\right]\rangle \simeq  \mathbb{Z}/p\mathbb{Z}.\]
Then the derived ideal and annihilator of $A$ are given by
\begin{align*}
   A*A =  \langle \left[\begin{smallmatrix}1\\0\end{smallmatrix}\right]\rangle \times C,\quad
      \Ann(A) = \langle \left[\begin{smallmatrix}1\\0\end{smallmatrix}\right]\rangle \times \{1\},
  \end{align*}
 and we have the non-equality $(A*A)*\Ann_2(A)\neq 1$.
\end{prop}
\begin{proof}The two equalities follow from (\ref{A*A}), (\ref{Ann(A)}), and the fact that
\[\mathrm{Im}(Q-I_2) =  \ker(Q-I_2)= \langle \left[\begin{smallmatrix}1\\0\end{smallmatrix}\right]\rangle\mbox{ }\]for all $Q \in \mathrm{Im}(\phi)$ with $Q\neq I_2$.

To show the non-equality, consider the element
\[ a = (\left[\begin{smallmatrix}0\\1\end{smallmatrix}\right],1) \in Z(A,\cdot).\]
For any $x_1,x_2\in \mathbb{F}_p$ and $y\in C$, it follows from (\ref{*B}) that
\begin{align}\notag
(\left[\begin{smallmatrix}0\\1\end{smallmatrix}\right],1) * (\left[\begin{smallmatrix}x_1\\x_2\end{smallmatrix}\right],y) & = ((\phi_1-I_2)\left[\begin{smallmatrix}x_1\\x_2\end{smallmatrix}\right], 1*y) = (\left[\begin{smallmatrix}0\\0\end{smallmatrix}\right],1)\\\label{neq1}
(\left[\begin{smallmatrix}x_1\\x_2\end{smallmatrix}\right],y) *(\left[\begin{smallmatrix}0\\1\end{smallmatrix}\right],1) & = ((\phi_y-I_2)\left[\begin{smallmatrix}0\\1\end{smallmatrix}\right], y*1) = (\left[\begin{smallmatrix}\gamma\\0\end{smallmatrix}\right],1)
\end{align}
for some $\gamma\in \mathbb{F}_p$. They both lie in $\Ann(A)$ and thus $a\in \Ann_2(A)$. Taking 
\[ z:= (\left[\begin{smallmatrix}x_1\\x_2\end{smallmatrix}\right],y) \mbox{ with }x_2=0\mbox{ and }\phi_y = \left[ \begin{smallmatrix}1 & 1\\0&1\end{smallmatrix}\right], \]
we see that $\gamma=1$ in (\ref{neq1}) and so $z*a \neq 1$. Since $a\in \Ann_2(A)$ and $z\in A*A$, this yields $(A*A)*\Ann_2(A)\neq 1$, as desired.
\end{proof}

For $n=3$, we need to take $p$ to be odd, but the situation is very similar.

\begin{prop}\label{prop2} 
Let $C = (C,\cdot,\circ)$ be a perfect skew brace with $\Ann(C) = 1$ and let $A = \mathbb{F}_p^3 \rtimes_\phi C$, where $p$ is an odd prime and $\phi : (C,\circ) \longrightarrow \mathrm{GL}_3(\mathbb{F}_p)$ is a homomorphism with
\[ \mathrm{Im}(\phi)=\left\langle  \left[\begin{smallmatrix}
1 & 0 & 1 \\
0 & 1 & 0\\
0 &0 &1
\end{smallmatrix}\right]\right\rangle \times
\left\langle\left[
\begin{smallmatrix}
1 & 1 & 0\\
0 & -1 &0 \\
0 & 0 & 1
\end{smallmatrix}
\right]\right\rangle
\simeq \mathbb{Z}/p\mathbb{Z}\times \mathbb{Z}/2\mathbb{Z}.\]
Then the derived ideal and annihilator of $A$ are given by
\begin{align*}
  A*A =  \left\langle \left[\begin{smallmatrix}1\\0\\0\end{smallmatrix}\right],\left[\begin{smallmatrix}0\\1\\0\end{smallmatrix}\right]\right\rangle \times C,\quad
   \Ann(A) = \left\langle \left[\begin{smallmatrix}1\\0\\0\end{smallmatrix}\right]\right\rangle \times \{1\},
\end{align*}
and we have the non-equality $(A*A)*\Ann_2(A)\neq 1$. 
\end{prop}

\begin{proof}  The claim for $\Ann(A)$ holds by (\ref{Ann(A)}) and the fact that
\begin{align*}
\ker\left( \left[\begin{smallmatrix}
1 & 0 & 1 \\
0 & 1 & 0\\
0 &0 &1
\end{smallmatrix}\right]-I_3\right) & = \left\langle\left[\begin{smallmatrix}1\\0\\0\end{smallmatrix}\right] ,\left[\begin{smallmatrix}0\\1\\0\end{smallmatrix}\right]\right\rangle,\\
\ker\left( \left[\begin{smallmatrix}
1 & 1 & 0 \\
0 & -1 & 0\\
0 &0 &1
\end{smallmatrix}\right]-I_3\right) & = \left\langle\left[\begin{smallmatrix}1\\0\\0\end{smallmatrix}\right],\left[\begin{smallmatrix}0\\0\\1\end{smallmatrix}\right] \right\rangle.
\end{align*}
Here, once we regard elements of $\mathrm{Fix}_{\phi}(\mathbb{F}_p^3)$ as the fixed points of $\mathrm{Im}(\phi)$, it is clear that it suffices to consider the generators of $\mathrm{Im}(\phi)$. Now, observe that
\begin{equation}\label{elements}
\mathrm{Im}(\phi) = \left\{\left[\begin{smallmatrix}
1 & 0 & \gamma \\
0 & 1 & 0\\
0 &0 &1
\end{smallmatrix}\right],\, 
\left[\begin{smallmatrix}
1 & 1 & \gamma\\
0 & -1 & 0\\
0 & 0 & 1
\end{smallmatrix}\right]:\gamma\in\mathbb{F}_p\right\}.\end{equation}
The claim for $A*A$ then follows from (\ref{A*A}) and the fact that
\begin{align*}
\mathrm{Im}\left( \left[\begin{smallmatrix}
1 & 0 & \gamma\\
0 & 1 & 0\\
0 &0 &1
\end{smallmatrix}\right]-I_3\right) & =\left\langle\left[\begin{smallmatrix}\gamma\\0\\0\end{smallmatrix}\right]\right\rangle,\\
\mathrm{Im}\left( \left[\begin{smallmatrix}
1 & 1 & \gamma\\
0 & -1 & 0\\
0 &0 &1
\end{smallmatrix}\right]-I_3\right) & =\left\langle\left[\begin{smallmatrix}1\\-2\\0\end{smallmatrix}\right],\left[\begin{smallmatrix}\gamma\\0\\0\end{smallmatrix}\right]\right\rangle,
\end{align*}
where $2\neq 0$ because $p$ is assumed to be odd. 

To show the non-equality, consider the element
\[a:=\left(\left[\begin{smallmatrix}0\\0\\1\end{smallmatrix}\right],1\right)\in Z(A,\cdot).\]
Observe that (\ref{elements}) implies
\[ \forall y \in C:\, (\phi_y - I_3)\left[\begin{smallmatrix}0\\0\\1\end{smallmatrix}\right] \in \left\langle \left[\begin{smallmatrix}1\\0\\0\end{smallmatrix}\right]\right\rangle.\]
For any $x_1,x_2,x_3\in \mathbb{F}_p$ and $y\in C$, we then see from (\ref{*B}) that
\begin{align}\notag
\left(\left[\begin{smallmatrix}0\\0\\1\end{smallmatrix}\right],1\right) * \left(\left[\begin{smallmatrix}x_1\\x_2\\x_3\end{smallmatrix}\right],y\right) & = \left((\phi_1-I_3) \left[\begin{smallmatrix}x_1\\x_2\\x_3\end{smallmatrix}\right], 1*y\right) = \left(\left[\begin{smallmatrix}0\\0\\0\end{smallmatrix}\right],1\right)\\\label{neq1'}
\left(\left[\begin{smallmatrix}x_1\\x_2\\x_3\end{smallmatrix}\right],y\right) *\left(\left[\begin{smallmatrix}0\\0\\1\end{smallmatrix}\right],1\right) & = \left((\phi_y-I_3)\left[\begin{smallmatrix}0\\0\\1\end{smallmatrix}\right], y*1\right) = \left(\left[\begin{smallmatrix}\gamma \\0\\0\end{smallmatrix}\right],1\right)
\end{align}
for some $\gamma\in\mathbb{F}_p$. They both lie in $\Ann(A)$ and thus $a\in \Ann_2(A)$. Taking 
\[ z:=\left(\left[\begin{smallmatrix}x_1\\x_2\\x_3\end{smallmatrix}\right],y\right)\mbox{ with }x_3=0\mbox{ and }\phi_y = \left[\begin{smallmatrix} 1 & 0 & 1 \\ 0 & 1 & 0\\ 0 & 0 & 1 \end{smallmatrix}\right],\]
we see that $\gamma=1$ in (\ref{neq1'}) and so $z*a \neq 1$. Since $a\in \Ann_2(A)$ and $z\in A*A$, this yields $(A*A)*\Ann_2(A)\neq 1$, as desired.\end{proof}

Note that the skew braces constructed in Propositions \ref{prop1} and \ref{prop2} are not perfect. To obtain a perfect skew brace without losing the desired condition (\ref{neq}), we shall use $B =(\mathbb{F}_2^4,+,+)$. The suitable subgroup $\mathrm{Im}(\phi)$ of $\mathrm{GL}_4(\mathbb{F}_2)$ that we found is isomorphic to the symmetric group $S_4$, and we acknowledge the use of \textsc{magma} \cite{Magma} in the search of this subgroup.

\begin{prop}\label{prop3}Let $C = (C,\cdot,\circ)$ be a perfect skew brace with $\Ann(C) = 1$ and let $A = \mathbb{F}_2^4 \rtimes_\phi C$, where $\phi : (C,\circ) \longrightarrow \mathrm{GL}_4(\mathbb{F}_2)$ is a homomorphism with
\begin{align*}
 \mathrm{Im}(\phi)
 & =
 \left\langle  \left[\begin{smallmatrix}
0 & 1 &1 & 1\\
0 & 1 & 0 & 0\\
1 & 0 & 0 & 1\\
0 &1 & 0 & 1
 \end{smallmatrix}\right],
\left[\begin{smallmatrix}
1 & 1 & 0 & 0\\
0 & 1& 0 & 0\\
1 & 1 & 0 & 1\\
1 & 0 & 1 &0
 \end{smallmatrix}\right]\right\rangle
 \rtimes
 \left\langle  \left[\begin{smallmatrix}
1 & 0 & 0 & 0\\
1 & 1 & 1 & 1\\
0 & 0 & 1 & 0 \\
 1 & 1 & 1 & 0
 \end{smallmatrix}\right],\left[
\begin{smallmatrix}
0 & 0 & 1 & 0\\
0 & 1 & 0 & 0 \\
1 & 0 & 0 & 0\\
0 & 1 & 0 & 1
\end{smallmatrix}
\right]\right\rangle
\simeq S_4.\end{align*}
Then the derived ideal and annihilator of $A$ are given by
\[ A*A = A,\quad \Ann(A) = \left\langle \left[\begin{smallmatrix}1\\0\\1\\0\end{smallmatrix}\right]\right\rangle \times \{1\},\]
and we have the non-equality $A*\Ann_2(A)\neq1$.\end{prop}

\begin{proof} The claim for $\Ann(A)$ holds by (\ref{Ann(A)}) and the fact that 
\begin{align*}\label{kernel}
\ker\left( \left[\begin{smallmatrix}
0 & 1 &1 & 1\\
0 & 1 & 0 & 0\\
1 & 0 & 0 & 1\\
0 &1 & 0 & 1
\end{smallmatrix}\right]-I_4\right) & = \left\langle\left[\begin{smallmatrix}1\\0\\1\\0\end{smallmatrix}\right] ,\left[\begin{smallmatrix}1\\0\\0\\1\end{smallmatrix}\right]\right\rangle,\\\notag
\ker\left( \left[\begin{smallmatrix}
1 & 1 & 0 & 0\\
0 & 1& 0 & 0\\
1 & 1 & 0 & 1\\
1 & 0 & 1 &0
\end{smallmatrix}\right]-I_4\right) & = \left\langle\left[\begin{smallmatrix}1\\0\\1\\0\end{smallmatrix}\right] ,\left[\begin{smallmatrix}1\\0\\0\\1\end{smallmatrix}\right]\right\rangle,\\\notag
\ker\left( \left[\begin{smallmatrix}
1 & 0 & 0 & 0\\
1 & 1 & 1 & 1\\
0 & 0 & 1 & 0 \\
 1 & 1 & 1 & 0
\end{smallmatrix}\right]-I_4\right) & = \left\langle\left[\begin{smallmatrix}1\\0\\1\\0\end{smallmatrix}\right] ,\left[\begin{smallmatrix}1\\0\\0\\1\end{smallmatrix}\right]\right\rangle,\\\notag
\ker\left( \left[\begin{smallmatrix}
0 & 0 & 1 & 0\\
0 & 1 & 0 & 0 \\
1 & 0 & 0 & 0\\
0 & 1 & 0 & 1
\end{smallmatrix}\right]-I_4\right) & = \left\langle\left[\begin{smallmatrix}1\\0\\1\\0\end{smallmatrix}\right] ,\left[\begin{smallmatrix}0\\0\\0\\1\end{smallmatrix}\right]\right\rangle.
\end{align*}
Here, again we only need to look at the generators of $\mathrm{Im}(\phi)$ because $\mathrm{Fix}_\phi(\mathbb{F}_2^4)$ is simply the set of  fixed points of $\mathrm{Im}(\phi)$. The claim for $A*A =A$ holds by (\ref{A*A}) because we have
\begin{align*}
\mathrm{Im}\left( \left[\begin{smallmatrix}
1 & 1 & 0 & 0\\
0 & 1& 0 & 0\\
1 & 1 & 0 & 1\\
1 & 0 & 1 &0
\end{smallmatrix}\right]-I_4\right) & \ni \left[\begin{smallmatrix}1\\0\\1\\0\end{smallmatrix}\right],\left[\begin{smallmatrix}0\\0\\1\\1\end{smallmatrix}\right] ,\\
\mathrm{Im}\left( \left[\begin{smallmatrix}
1 & 0 & 0 & 0\\
1 & 1 & 1 & 1\\
0 & 0 & 1 & 0 \\
 1 & 1 & 1 & 0
\end{smallmatrix}\right]-I_4\right) & \ni \left[\begin{smallmatrix}0\\1\\0\\1\end{smallmatrix}\right],\left[\begin{smallmatrix}0\\0\\0\\1\end{smallmatrix}\right] ,
\end{align*}
and the four vectors on the right are linearly independent.

To show that non-equality, consider the element
\[ a:=  \left(\left[\begin{smallmatrix}1\\0\\0\\1\end{smallmatrix}\right],1\right) \in Z(A,\cdot).\]
Let $U$ denote the subspace of $\mathbb{F}_2^4$ generated by $(1,0,1,0)$. Observe that
\begin{align*}
 \left(\left[\begin{smallmatrix}
0 & 0 & 1 & 0 \\
0 & 1 & 0 & 0\\
1 & 0 & 0 & 0\\
0 & 1 & 0 & 1
\end{smallmatrix}\right]-I_4\right)
\left[\begin{smallmatrix} 1 \\ 0 \\ 0 \\ 1\end{smallmatrix}\right] & = 
\left[\begin{smallmatrix} 1 \\ 0 \\ 1 \\ 0\end{smallmatrix}\right],\quad\mbox{and}\quad(Q - I_4)\left[\begin{smallmatrix} 1 \\ 0 \\ 0 \\ 1\end{smallmatrix}\right] =\left[\begin{smallmatrix}0 \\ 0 \\ 0 \\ 0\end{smallmatrix}\right] 
\end{align*}
for the other three generators $Q$ of $\mathrm{Im}(\phi)$. Since $U$ is fixed by all four of the generators, the above implies that
\[ \forall y\in C: \phi_y\left(\left[\begin{smallmatrix} 1 \\ 0 \\ 0 \\ 1\end{smallmatrix}\right]\right)\equiv \left[\begin{smallmatrix} 1 \\ 0 \\ 0 \\ 1\end{smallmatrix}\right]\hspace{-5mm}\pmod{U},\mbox{ that is }
(\phi_y-I_4)\left[\begin{smallmatrix} 1 \\ 0 \\ 0 \\ 1\end{smallmatrix}\right]\in U.\]
For any $x_1,x_2,x_3,x_4\in \mathbb{F}_2$ and $y\in C$, we then see from (\ref{*B}) that
 \begin{align}\notag
\left(\left[\begin{smallmatrix}1\\0\\0\\1\end{smallmatrix}\right],1\right) * \left(\left[\begin{smallmatrix}x_1\\x_2\\x_3\\x_4\end{smallmatrix}\right],y\right) & = \left((\phi_1-I_4) \left[\begin{smallmatrix}x_1\\x_2\\x_3\\x_4\end{smallmatrix}\right], 1*y\right) = \left(\left[\begin{smallmatrix}0\\0\\0\\0\end{smallmatrix}\right],1\right)\\\label{*}
\left(\left[\begin{smallmatrix}x_1\\x_2\\x_3\\x_4\end{smallmatrix}\right],y\right) *\left(\left[\begin{smallmatrix}1\\0\\0\\1\end{smallmatrix}\right],1\right) & = \left((\phi_y-I_4)\left[\begin{smallmatrix}1\\0\\0\\1\end{smallmatrix}\right], y*1\right) = \left(\left[\begin{smallmatrix}\gamma\\0\\\gamma\\0\end{smallmatrix}\right],1\right)
\end{align}
for some $\gamma \in \mathbb{F}_2$. They both lie in $\Ann(A)$ and thus $a\in \Ann_2(A)$. Since we can choose $y\in C$ such that $\gamma=1$ in (\ref{*}), it follows that $z*a\neq 1$ for some $z\in A$. Since $a\in \Ann_2(A)$, we get $A*\Ann_2(A)\neq 1$, as desired.
\end{proof}

As one can see from the proof of Proposition \ref{prop3}, the main idea is to take
\[ \mathrm{Im}(\phi)=\langle M_1,\dots,M_d\rangle,\]
where we choose the matrices $M_1,\dots,M_d\in\mathrm{GL}_n(\mathbb{F}_p)$ to be  such that
\begin{enumerate}[(1)]
\item the columns of $M_1-I_n,\dots,M_d-I_n$ generate $\mathbb{F}_p^n$;
\item there exists $\vec{v}\in \mathbb{F}_p^n$ lying outside of
\[ U:= \bigcap_{i=1}^{d} \ker(M_i-I_n) \]
but gets mapped into $U$ under every $M_1-I_n,\dots,M_d-I_n$.
\end{enumerate}
These conditions, respectively, ensure that for the skew brace $A=\mathbb{F}_p^n\rtimes_\phi C$ (with $C$ perfect and $\Ann(C)=1$) constructed, we have:
\begin{enumerate}[(1)]
\item $A$ is perfect (recall (\ref{A*A}));
\item $(\vec{v},1)\in \Ann_2(A)\setminus \Ann(A)$ (recall (\ref{Ann(A)}) and (\ref{*B})).\end{enumerate}
We can produce more candidates for $\mathrm{Im}(\phi)$ other than the one in Proposition \ref{prop3}. However, there might not exist a perfect skew brace $C$ with $\Ann(C)=1$ such that $(C,\circ)$ has the candidates as quotients. We have taken $\mathrm{Im}(\phi)\simeq S_4$ because thanks to the \texttt{YangBaxter} package in \texttt{GAP} \cite{GAP}, we have the following examples. Here $\texttt{SmallSkewbrace(n,k)}$ denotes the skew brace whose ID is $(n,k)$ in the database of \texttt{GAP}, and similarly for $\texttt{SmallGroup(n,k)}$.

\begin{example}\label{example1} According to \texttt{GAP}, the skew brace
\[ C = (C,\cdot,\circ)= \texttt{SmallSkewbrace(24,853)}\]
is perfect with $(C,\cdot)\simeq \mathbb{F}_3\times \mathbb{F}_2^3$ and $(C,\circ)\simeq S_4$. We have $\mathrm{Ann}(C)=1$ since $(C,\circ)$ has trivial center. This skew brace $C$ is among the simple braces that were constructed in \cite[Sections 6 and 7]{Bac}. In our special case, one finds that the lambda map of $C$ is given by
\[ \lambda_{(v,x_1,x_2,x_3)} = \left(
(-1)^{x_3-x_1x_2} ,\begin{bmatrix}
0 & 1 & 0\\
1 & 1 & 0\\
0 & 1 & 1
\end{bmatrix}^v
\left[
\begin{array}{cc}
M(x_1,x_2,x_3) & \begin{array}{c}0\\0\end{array}\\
\epsilon(v,x_1,x_2,x_3) & 1
\end{array}
\right]
\right),\]
where we define
\begin{align*}
M(x_1,x_2,x_3) & = \begin{bmatrix}
0 & 1 \\ 1 & 0
\end{bmatrix}^{x_3-x_1x_2}\\
\epsilon(v,x_1,x_2,x_3) & = \left[ x_1 \,\  x_2\right]
\begin{bmatrix}
0 & 1\\ 1 & 1 \end{bmatrix}^v
\begin{bmatrix}0 & 1 \\ 1 & 0 \end{bmatrix}^{1 + x_3-x_1x_2},
\end{align*}
for all $v\in\mathbb{F}_3$ and $x_1,x_2,x_3\in \mathbb{F}_2$. In the notation of \cite[Theorem 6.3]{Bac}, we are taking $Q(x_1,x_2)= x_1x_2,\, \gamma = -1,\, C = \left[ \begin{smallmatrix}
0 & 1\\ 1 & 1
\end{smallmatrix}\right],\, F = \left[\begin{smallmatrix}
0 & 1 \\ 1 & 0 
\end{smallmatrix}\right]$, and $z = \left[\begin{smallmatrix}0 & 1 \end{smallmatrix}\right]$. The function $q(x_1,x_2,x_3) = x_3-Q(x_1,x_2)$ is determined by $Q$, while the matrix $B$ represents the bilinear form $(\vec{x},\vec{y})\mapsto Q(\vec{x}+\vec{y}) - Q(\vec{x}) - Q(\vec{y})$ associated to $Q$, which is given by $\left[\begin{smallmatrix} 0 & 1\\ 1 & 0 \end{smallmatrix}\right]$ in this case. Finally, we note that the $c(\vec{x})$ there denotes the linear transformation on $\mathbb{F}_2^2$ induced by $C$.
\end{example}

\begin{example}\label{example2} According to \texttt{GAP}, the skew brace
\[ C = (C,\cdot,\circ)= \texttt{SmallSkewbrace(72,1483)}\]
is perfect with $(C,\cdot)\simeq C_9\times\mathbb{F}_2^3$ and $(C,\circ)\simeq \texttt{SmallGroup}(72,15)$, which is the unique non-split extension of $C_3$ and $S_4$. We have $\mathrm{Ann}(C)=1$ since $(C,\circ)$ has trivial center. This skew brace $C$ is a socle extension of that in Example \ref{example1}, in the sense that $C/\mathrm{Soc}(C)$ is isomorphic to \texttt{SmallSkewbrace(24,853)}.
\end{example}

Therefore, for the skew braces $C = (C,\cdot,\circ)$ in Examples \ref{example1} and \ref{example2}, there exists a homomorphism $\phi:(C,\circ)\longrightarrow \mathrm{GL}_4(\mathbb{F}_2)$ such that $\mathrm{Im}(\phi)\simeq S_4$ is given as in Proposition \ref{prop3}. The skew brace $A = \mathbb{F}_2^4\rtimes_\phi C$  is then perfect but
\[ \Ann(A/\Ann(A)) \neq 1.\]
This gives counterexamples to the analog of Gr\"{u}n's lemma. We remark that the $A$ here is in fact a brace because $(C,\cdot)$ is abelian for both examples.
 
\begin{remark} In the setting of Proposition \ref{sd prop}, similar to (\ref{Ann(A)}) we also have
\[ \Ann(A) = \mathrm{Fix}_\phi(B) \times \{1\}\mbox{ when }\ker(\phi)=1.\]
One can try to produce examples with $(A*A)* \Ann_2(A)\neq 1$ by taking $\phi$ to be injective rather than requiring $\Ann(C)=1$.
\end{remark}


\section*{Acknowledgments} 

This research is supported by JSPS KAKENHI Grant Number 24K16891.

The author would also like to thank the referee for helpful comments.

\end{document}